\documentclass[12pt,a4paper]{amsart}

\usepackage{amssymb}
\usepackage{amsthm}
\usepackage{amsmath}
\usepackage{graphicx}
\usepackage{enumerate}
\usepackage{xcolor}
\usepackage{tikz}
\usepackage{graphbox}
\usepackage{orcidlink}
\usepackage{mathrsfs}  
\usepackage{hyperref}
\hypersetup{colorlinks=true,allcolors=blue}
\usepackage[margin=2.5cm]{geometry}
\usepackage{mathabx}

\newtheorem{theorem}{Theorem}[section]
\newtheorem{proposition}[theorem]{Proposition}
\newtheorem{corollary}[theorem]{Corollary}
\newtheorem{lemma}[theorem]{Lemma}

\theoremstyle{definition}
\newtheorem{example}[theorem]{Example}
\newtheorem{remark}[theorem]{Remark}

\newcommand{\kc}{\mathcal{C}}

\title[Independence vs. marginal independence]{How far are $d$-dimensional copulas 
with uniform $(d-1)$-marginals from (total) independence?}

\author[D. Kokol Bukovšek]{Damjana Kokol Bukovšek \orcidlink{0000-0002-0098-6784} }
\address{University of Ljubljana, School of Economics and Business, and Institute of Mathematics, Physics and Mechanics, Ljubljana, Slovenia}
\email{Damjana.Kokol.Bukovsek@ef.uni-lj.si}

\author[N. Stopar]{Nik Stopar \orcidlink{0000-0002-0004-4957} } 
\address{University of Ljubljana, Faculty of Civil and Geodetic Engineering, University of Ljubljana, Faculty of Mathematics and Physics, and Institute of Mathematics, Physics and Mechanics, Ljubljana, Slovenia}
\email{Nik.Stopar@fgg.uni-lj.si}

\author[W. Trutschnig]{Wolfgang Trutschnig \orcidlink{0000-0002-7131-1944} }
\address{University of Salzburg, Working group for Data Science, Statistics, Stochastics, Department for Artificial Intelligence and Human Interfaces, Salzburg, Austria}
\email{Wolfgang.Trutschnig@plus.ac.at}


\begin{document}

\begin{abstract}
We consider the family $\kc_d^{\Pi_{d-1}}$ of all $d$-dimensional copulas 
whose $(d-1)$-dimensional marginals are all equal to the $(d-1)$-dimensional 
product copula $\Pi_{d-1}$ and tackle the natural question, `how far away' from the 
$d$-dimensional product copula $\Pi_d$ elements in $\kc_d^{\Pi_{d-1}}$ can be.
We provide definitive answers for both, the uniform metric $d_\infty$ as well 
as the stronger, conditioning-based metric $D_1$. The established results clearly indicate that the family $\kc_d^{\Pi_{d-1}}$
is larger than one might expect.  
\end{abstract}

\maketitle

\section{Introduction}

In the paper \cite{NeUb12}, entitled `How close are pairwise and mutual independence?' 
the authors derived sharp pointwise bounds for the fa\-mily of 
all three-dimensional copulas with uniform bivariate marginals and provided 
an asymptotic result for the general $d$-dimensional setting with  $d\geq3$. 
Moving from bivariate uniform to $(d-1)$-dimensional uniform marginals, we here consider 
$d \geq 3$ and study the class
$\kc_d^{\Pi_{d-1}}$ of all $d$-dimensional copulas 
whose $(d-1)$-dimensional marginals all coincide with the \mbox{$(d-1)$}-dimensional 
product copula $\Pi_{d-1}$. 
The intuitive conjecture, formulated in \cite{NeUb12}, saying  
that pairwise independence should restrict variability/flexibility as well as 
the maximum possible distance to mutual (total) independence modelled by 
$\Pi_d$, naturally applies even more for the family $\kc_d^{\Pi_{d-1}}$. 
As our results show, however, $\kc_d^{\Pi_{d-1}}$ is much larger and versatile  
than one might expect -- 
regardless of whether we consider the standard uniform metric $d_\infty$ or 
the stronger metric $D_1$ (originally introduced in \cite{tru11} in the bivariate setting and generalized in \cite{GJT} to the multivariate case
to quantify the extent of dependence of a random variable $Y$ on a random vector $\mathbf{X}$).
As our main results we will prove that
$$
\sup\Big\{d_\infty(A,\Pi_d):\, A \in \kc_d^{\Pi_{d-1}}\Big\}=\tfrac{1}{2^d}
$$
as well as 
$$
\sup\Big\{D_1(A,\Pi_d):\, A \in \kc_d^{\Pi_{d-1}}\Big\}=\tfrac{1}{3}
$$
holds for every $d \geq 3$ and provide examples showing that both suprema are attained.

The remainder of this note is organized as follows: Section 2 gathers notation and 
preliminaries, Section 3 presents and proves the main results. Several 
examples and figures illustrate the underlying ideas. 

\section{Notation and Preliminaries}\label{sec:2}
Throughout the paper bold symbols will denote vectors. For every fixed 
index/coor\-dinate $i_0 \in [d] := \{1,\ldots,d\}$ and every $\mathbf{x} \in \mathbb{R}^d$ we will write
$\mathbf{x}_{-i_0}=(x_1,\ldots,x_{i_0-1},x_{i_0+1},\ldots,x_d)$. 
Moreover, for every vector $\mathbf{j}=(j_1,\ldots,j_l)$ of at most $l \leq d-1$ pairwise 
distinct indices in $[d]$ we will denote by $\pi_\mathbf{j}: [0,1]^d \rightarrow [0,1]^l$ 
the projection onto the coordinates in $\mathbf{j}$, i.e., 
$\pi_{\mathbf{j}}(x_1,\ldots,x_d)=(x_{j_1},\ldots,x_{j_l})$. 
To simplify notation we will also write $\pi_{1:l}$ for the projection onto the first $l$ coordinates
as well as $\pi_{-j_0}$ for the projection onto the coordinates $(1,\ldots,j_0-1,j_0+1,\ldots,d)$. 

As already mentioned before, $\kc_d$ will denote the class of all $d$-dimensional copulas 
for $d \geq 2$, and $\Pi_d$ the $d$-dimensional product copula. 
We will write $\mathcal{P}_{\mathcal{C}_d}$ for the class of all $d$-stochastic measures
(i.e., probability measures for which the corresponding distribution function restricted to $[0,1]^d$
is a $d$-dimensional copula). For every $A\in \kc_d$ the correspon\-ding $d$-stochastic measure 
will be denoted by $\mu_A \in \mathcal{P}_{\mathcal{C}_d}$; moreover, 
$\lambda_d=\mu_{\Pi_d}$ will denote the $d$-dimensional Lebesgue measure on $[0,1]^d$.   
For a random vector $\mathbf{X}=(X_1,\ldots,X_d)$ on a probability space 
$(\Omega,\mathcal{A},\mathbb{P})$ and $A \in \kc_d$ we will write 
$\mathbf{X}\sim A$ if, and only if, $A$ is the restriction of the distribution function of
$\mathbf{X}$ to $[0,1]^d$ (or, equivalently, if the distribution $\mathbb{P}^\mathbf{X}$ of $\mathbf{X}$ coincides with $\mu_A$). Notice that 
for $\mathbf{X}=(X_1,\ldots,X_d)\sim A$ we have that each $X_i$ is uniformly distributed on $[0,1]$. 

For every $A \in \kc_d$ and every vector $\mathbf{j}=(j_1,\ldots,j_l)$ of $l \le d-1$ distinct indices in $[d]$ we will let $A^\mathbf{j}$ denote the marginal 
copula corresponding to $\mathbf{j}$, i.e., the copula corresponding to the push-forward 
$\mu_{A^\mathbf{j}}=\mu_A^{\pi_{\mathbf{j}}}$ of $\mu_A$ via $\pi_{\mathbf{j}}$; 
$A^{1:l}$ denotes the marginal copula of the first $l$ coordinates and $A^{-j_0}$ the 
marginal copula of all coordinates except coordinate $j_0$.  
For further properties of copulas and $d$-stochastic measures we refer the reader to \cite{DuSe}, and \cite{Nels}. 

For a topological space $T$ we will denote the Borel $\sigma$-algebra on $T$ by $\mathcal{B}(T)$.
We call a map 
$K: \mathbb{R}^{d-1}\times\mathcal{B}(\mathbb{R}) \rightarrow [0,1]$ a
$(d-1)$-\emph{Markov kernel} from 
$\mathbb{R}^{d-1}$ to $\mathbb{R}$, if the function 
$\mathbf{x}\mapsto K(\mathbf{x},E)$ is 
$\mathcal{B}(\mathbb{R}^{d-1})$-$\mathcal{B}(\mathbb{R})$-measurable for every fixed 
$E\in\mathcal{B}(\mathbb{R})$ and the map $E\mapsto K(\mathbf{x},E)$ is a probability measure 
on $\mathcal{B}(\mathbb{R})$ for every $\mathbf{x}\in\mathbb{R}^{d-1}$. 
Given a random variable $Y$ and a $(d-1)$-dimensional random vector 
$\mathbf{X}$ on a joint probability space $(\Omega,\mathcal{A},\mathbb{P})$, a Markov kernel 
$K$ is called a regular conditional distribution of
$Y$ given $\mathbf{X}$, if (and only if) for every set 
$E \in \mathcal{B}(\mathbb{R})$ the identity 
$$
K(\mathbf{X}(\omega), E) = \mathbb{E}(\mathbf{1}_E \circ Y | \mathbf{X})(\omega)
$$
holds for $\mathbb{P}$-almost every $\omega \in \Omega$.
It is a well-known fact, see \cite{Kallenberg}, that for each pair $(\mathbf{X}, Y)$ 
such a regular conditional distribution $K$ of $Y$ given $\mathbf{X}$ exists and that it is 
unique for $\mathbb{P}^{\mathbf{X}}$-almost every $\mathbf{x} \in \mathbb{R}^{d-1}$.
For $(\mathbf{X}, Y) \sim A$ we will let 
$K_{A}:[0,1]^{d-1} \times \mathcal{B}([0,1]) \to [0,1]$ denote (a version of) the corresponding 
conditional distribution of $Y$ given $\mathbf{X}$ (i.e., we view 
$K_A$ directly as a mapping from $[0,1]^{d-1} \times \mathcal{B}([0,1])$ to 
$[0,1]$); $K_A$ will simply be referred to as 
(a version of) the $(d-1)$-\emph{Markov kernel} of the copula $A$. 

For every $G \subseteq [0,1]^d$ and $\mathbf{x} \in [0,1]^{d-1}$ 
define the $\mathbf{x}$-section $G_\mathbf{x}$ of $G$ by $G_{\mathbf{x}}:=\{y\in [0,1]: 
(\mathbf{x},y) \in G\}\in\mathcal{B}([0,1])$. 
Applying \emph{disintegration} of $\mu_A$ into the marginal $\mu_{A^{1:d-1}}$ and the 
$(d-1)$-Markov kernel $K_A$ of $A$,
 see \cite{Kallenberg}, Section 5, the following identity 
holds for all $G \in \mathcal{B}([0,1]^d)$:
\begin{align}\label{eq:DI}
	\mu_A(G) = \int_{\mathbb{I}^{d-1}} K_{A}(\mathbf{x},G_{\mathbf{x}})
	\, \mathrm{d}\mu_{A^{1:d-1}}(\mathbf{x}).
\end{align}
We say that a mapping $g: [0,1]^{d-1} \rightarrow [0,1]$ is $\mu_{A^{1:d-1}}$-$\lambda$ preserving, if $g$ is Borel measurable and fulfills 
that the push-forward $(\mu_{A^{1:d-1}})^g$ of $\mu_{A^{1:d-1}}$ via $g$ 
coincides with $\lambda$.
In accordance with \cite{GJT} we will call a copula $A \in \kc_d$ \emph{completely dependent} (on the first $(d-1)$ coordinates), if there exists some 
$\mu_{A^{1:d-1}}$-$\lambda$-preserving transformation $g$ such that
$K(\mathbf{x},F):=\mathbf{1}_F(g(\mathbf{x}))$ is a regular conditional distribution of $A$. 
For $(\mathbf{X},Y)\sim A$ we have that complete dependence of $A$ is equivalent to the
existence of some $\mu_{A^{1:d-1}}$-$\lambda$ preserving transformation $g$ such that
$Y=g(\mathbf{X})$ almost surely. 

In the sequel we will work with the family $\kc_d^{\Pi_{d-1}}$ of all $d$-dimensional 
copulas whose $(d-1)$-dimensional marginals are all equal to $\Pi_{d-1}$. 
Letting $d_\infty$ denote the uniform distance on $\kc_d^{\Pi_{d-1}}$ it is straightforward
to verify that $(\kc_d^{\Pi_{d-1}},d_\infty)$ is a convex compact set. 
Following \cite{GJT}, considering the so-called linkage family
$$
\hat{\kc}_d := \{A \in \kc_d: A^{1:d-1}=\Pi_{d-1}\} \supseteq \kc_d^{\Pi_{d-1}}
$$
and setting 
\begin{equation}
D_1(A,B)= \int_{[0,1]} \int_{[0,1]^{d-1}} \big\vert K_A(\mathbf{x},[0,y]) - K_B(\mathbf{x},[0,y]) \big\vert d\lambda_{d-1}(\mathbf{x})d\lambda(y)
\end{equation}
defines a metric on $\hat{\kc}_d$. The resulting metric space 
$(\hat{\kc}_d,D_1)$ is complete and separable, convergence w.r.t. $D_1$ implies 
convergence w.r.t. to $d_\infty$ but not vice versa. The metric $D_1$ was introduced for quantifying the extent of directed dependence (again see \cite{GJT} and the underlying bivariate construction in \cite{tru11}) and 
therefore (contrary to $d_\infty$) exhibits the property that 
$D_1(A,\Pi_d)$ is maximal if and only if
$A$ is completely dependent, and the maximum is given by $\frac{1}{3}$.  
Furthermore, the diameter of the metric space $(\hat{\kc}_d,D_1)$ is $\frac{1}{2}$. 

If $B = \displaystyle\bigtimes_{j \in [d]} [x_j,y_j] \subseteq [0,1]^d$ is a $d$-dimensional rectangle, its \emph{vertices} will be denoted by
$$\textup{ver}(B)=\displaystyle\bigtimes_{j \in [d]} \{x_j,y_j\}.$$
For every vertex $\mathbf{v} \in \textup{ver}(B)$ let $m(\mathbf{v})=|\{j\in[d] \colon v_j=x_j\}|$ and define the \emph{sign} of vertex $\mathbf{v}$ of rectangle $B$ by $\textup{sgn}_B(\mathbf{v})=(-1)^{m(\mathbf{v})}$. For every copula $A\in \kc_d$ we have
\begin{equation}\label{eq:mass.quasi}
 \mu_A(B) = \sum_{\mathbf{v} \in \textup{ver}(B)} \textup{sgn}_B(\mathbf{v})A(\mathbf{v}) \ge 0.  
\end{equation}
If $A$ is only a $d$-dimensional quasi-copula, then we will define $\mu_A(B)$ 
according to eq. (\ref{eq:mass.quasi}) -- in this case $\mu_A(B)$ may also be negative
and $\mu_A$ does not necessarily extend to a (signed) measure, see \cite{DFST16}.

\section{Main results}
Our main objective is to calculate the quantities
\begin{align}
\Delta_{d_\infty} &:= \sup\Big\{d_\infty(A,\Pi_d):\, A \in \kc_d^{\Pi_{d-1}}\Big\}, \\
\Delta_{D_1} &:= \sup\Big\{D_1(A,\Pi_d):\, A \in \kc_d^{\Pi_{d-1}}\Big\}, 
\end{align}
and to show that both suprema are attainable.
First we focus on the second one, and start with a simple lemma which is valid in the much 
more general setting of Haar measure on compact groups, see \cite{RuF}, and the proof of which 
we include for the sake of completeness. 
\begin{lemma}\label{lem:haar}
    Suppose that $X$ and $Y$ are independent random variables. 
    If $X$ is uniformly distributed on $[0,1]$,  
    then so is the random variable $Z$ given by 
    $$
    Z:= X +Y \,\, \textup{mod (1)}.
    $$
\end{lemma}
\begin{proof}
    For every $a \in [0,1)$ define the shift $r_a: [0,1] \rightarrow [0,1]$ by 
    $r_a(x)=(x+a) \, \textup{mod (1)}$. Then obviously $r_a$ preserves the Lebesgue 
    measure $\lambda$, i.e., 
    the push-forward $\lambda^{r_a}$ of $\lambda$ via $r_a$ coincides with $\lambda$.
    Let $E \in \mathcal{B}([0,1])$ be arbitrary but fixed. Then using independence of 
    $X$ and $Y$, Fubini's theorem, change of coordinates, and $\lambda^{r_y}=\lambda$ 
    (in this order) yields 
    \begin{align*}
    \mathbb{P}(Z &\in E) = \int_\Omega \mathbf{1}_E(Z) d\mathbb{P}  = 
                \int_{[0,1] \times \mathbb{R}} \mathbf{1}_{E}(r_y(x)) d \mathbb{P}^{X,Y}(x,y)\\
                &=\int_{[0,1]} \int_{\mathbb{R}} \mathbf{1}_{E}(r_y(x)) d \mathbb{P}^Y(y) 
                         d\mathbb{P}^X(x) = 
                  \int_{\mathbb{R}} \int_{[0,1]} \mathbf{1}_{E}(r_y(x)) d \mathbb{P}^X(x) 
                         d\mathbb{P}^Y(y) \\
                      &=  \int_{\mathbb{R}} \int_{[0,1]} \mathbf{1}_{E}(r_y(x)) 
                         d\lambda(x) d \mathbb{P}^Y(y)  = 
                         \int_{\mathbb{R}} \int_{[0,1]} \mathbf{1}_{E}(v) 
                         d\lambda^{r_y}(v) d \mathbb{P}^Y(y)\\ 
                      &= \int_{\mathbb{R}} \lambda^{r_y}(E)d\mathbb{P}^Y(y) 
                      = \int_{\mathbb{R}} \lambda(E)d\mathbb{P}^Y(y) = \lambda(E).    
    \end{align*}
   This shows that $\mathbb{P}^Z$ coincides with $\lambda$ on $[0,1]$ and the proof is complete. 
\end{proof}
Taking into account that the diameter of $(\hat{\kc}_d,D_1)$ is $\frac{1}{2}$ 
the following theorem shows that $\kc_d^{\Pi_{d-1}} \subseteq \hat{\kc}_d$ contains elements that 
are surprisingly `far away' from $\Pi_d$, in fact, as far away as any element of $\hat{\kc}_d$ can be (with respect to $D_1$).
\begin{theorem}\label{thm:mainD1}
For every $d \geq 3$ we have $\Delta_{D_1}=\frac{1}{3}$ and the supremum is attained.
\end{theorem}
\begin{proof}
    Considering the afore-mentioned properties of $D_1$, it suffices to show that the class 
    $ \kc_d^{\Pi_{d-1}}$ contains completely 
    dependent copulas, which can be done as follows.
    Let $X_1,\ldots,X_{d-1}$ denote independent, uniformly $[0,1]$-distributed random variables
    and set 
    \begin{equation}\label{eq:defX_q}
        X_d:= (X_1 + X_2 + \ldots + X_{d-1})  \,\, \textup{mod (1)}. 
    \end{equation}
    Then, as a direct consequence of Lemma \ref{lem:haar}, $X_d$ is uniformly distributed on 
    $[0,1]$ too. Letting $A$ denote the copula
    of the vector $\mathbf{X}=(X_1,\ldots,X_d)$ we obviously have that 
    $A \in \hat{\kc}_d$ and that $A$ is completely dependent. In fact, the vector 
    $\mathbf{X}$ even fulfills that each coordinate $X_j$ is a function of the other $(d-1)$ coordinates, i.e., a function of $\mathbf{X}_{-j}$. In other words, $A$ is completely 
    dependent in each direction.     
    To complete the proof it suffices to show that all $(d-1)$-dimensional marginals 
    of $A$ coincide with $\Pi_{d-1}$, which can be done as follows. 
    Considering disintegration (w.r.t. the coordinates $2,\ldots,d-1$) and again using  $\lambda^{r_a}=\lambda$,  
    for arbitrary $E_2,E_3,\ldots, E_d \in \mathcal{B}([0,1])$ we get
    \begin{align*}
        \mathbb{P}^{\mathbf{X}_{-1}} \Bigg(\bigtimes_{i=2}^d &E_i\Bigg) = 
        \mathbb{P}^{\mathbf{X}} \left([0,1] \times \bigtimes_{i=2}^d E_i\right) \\
        &= 
        \int_{\bigtimes_{i=2}^{d-1}E_i} \mathbb{P}\left(X_1 \in [0,1], r_{\sum_{i=2}^{d-1}x_i \,\, \textup{mod (1)}}(X_1) \in E_d \right)d\lambda_{d-2}(\mathbf{x}_{2:d-1}) \\
        &= \int_{\bigtimes_{i=2}^{d-1}E_i} \mathbb{P}\left(r_{\sum_{i=2}^{d-1}x_i \,\, \textup{mod (1)}}(X_1) \in E_d \right)d\lambda_{d-2}(\mathbf{x}_{2:d-1}) \\
        &= \int_{\bigtimes_{i=2}^{d-1}E_i} \lambda(E_d) \,d\lambda_{d-2}(\mathbf{x}_{2:d-1}) 
        = \prod_{i=2}^d \lambda(E_i) = \prod_{i=2}^d \mathbb{P}^{X_i}(E_i). 
        \end{align*}
    This shows $\mathbf{X}_{-1} \sim \Pi_{d-1}$. Replacing $j=1$ by any other index
    $j \in \{2,\ldots,d-1\}$ and proceeding analogously yields $\mathbf{X}_{-j} \sim \Pi_{d-1}$. 
    Since $\mathbf{X}_{-d} \sim \Pi_{d-1}$ by construction, the proof is complete. The left panel of Figure \ref{fig1} depicts a sample of the copula $A$ for 
    the case $d=3$. 
    \end{proof}
Our proof of Theorem \ref{thm:mainD1} establishes a slightly stronger result -- before
stating it, we fix some notation: $\Sigma_d$ denotes the set of all $d!$ permutations of 
$[d] = \{1,\ldots,d\}$; for every permutation $\varepsilon \in \Sigma_d$ and 
$A \in \kc_d^{\Pi_{d-1}}$ we let $A^\varepsilon$ denote 
the copula obtained from $A$ by permuting the coordinates according to $\varepsilon$, i.e., 
$A^\varepsilon$ is the copula of the vector $\mathbf{X}_{\varepsilon(\mathbf{i})}$ where
$\mathbf{X} \sim A$. 
\begin{corollary}
For every $d \geq 3$ we have 
$$
\sup\Big\{\min\{D_1(A^\varepsilon,\Pi_d): \, \varepsilon \in \Sigma_d\}: \, A \in \kc_d^{\Pi_{d-1}}\Big\} = \tfrac{1}{3}
$$
and the supremum is attained.
\end{corollary}

\begin{figure}[ht]
\begin{center}   
  \includegraphics[width=0.48\textwidth]{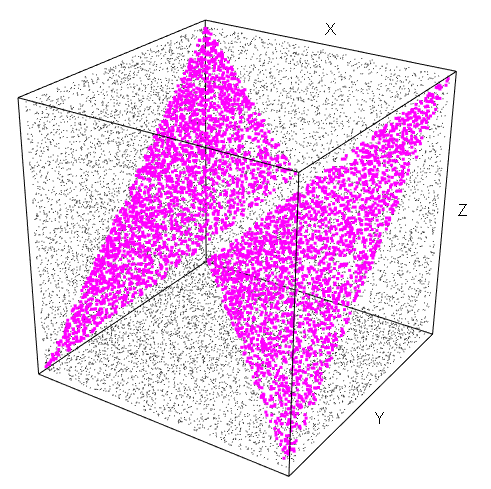}
  \includegraphics[width=0.48\textwidth]{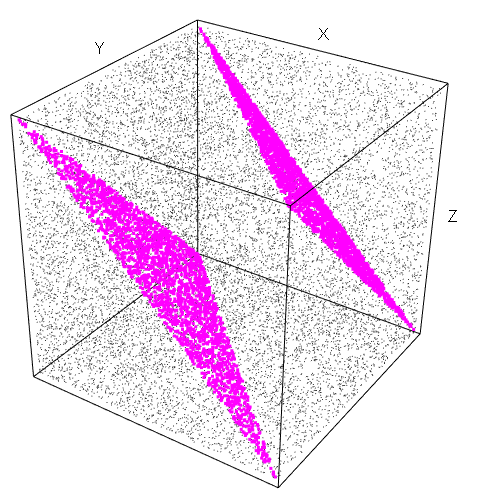}
\end{center}
\caption{Left panel: Sample of size $n=5.000$ of the copula $A \in \kc_3^{\Pi_2}$ (in dimension $d=3$) as 
    considered in the proof of Theorem 
    \ref{thm:mainD1} in magenta, bivariate marginal samples in gray (i.e., the 
    three-dimensional magenta points are projected onto the hyperplanes $h_1: z=0$,
    $h_2: y=1$, $h_3: x=0$). \\
    Right panel: Sample of the symmetric copula $B\in \kc_3^{\Pi_2}$ considered in the second 
    part of Remark \ref{rem:symm}.}
    \label{fig1}
\end{figure}

\begin{remark}\label{rem:symm}
    Notice that fixing $a \in (0,1)$, slightly changing the definition 
    of $X_d$ according to eq. (\ref{eq:defX_q}) to 
    \begin{equation*}
        X_d:= (X_1 + X_2 + \ldots + X_{d-1} + a)  \,\, \textup{mod (1)}, 
    \end{equation*}
    and proceeding like in the proof of Theorem \ref{thm:mainD1}
    results in another copula $A_a \in \kc_d^{\Pi_{d-1}}$ 
    which is completely dependent and which also fulfills 
    $D_1(A_a,\Pi_d)=\frac{1}{3}=\Delta_{D_1}$. 
    In other words, we can easily construct uncountably many elements in 
    $\kc_d^{\Pi_{d-1}}$ having $D_1$-distance $\frac{1}{3}$ to $\Pi_d$.
    Moreover, our construction can easily be modified in order to obtain symmetric 
    (exchangeable) copulas with the afore-mentioned maximization property
    w.r.t.~$D_1$. 
    In fact, replacing (\ref{eq:defX_q}) by 
    \begin{equation*}
        X_d:= 1 - \big((X_1 + X_2 + \ldots + X_{d-1})  \,\, \textup{mod (1)}\big), 
    \end{equation*}
    we obviously have $(\sum_{i=1}^d X_i) \,\,\textup{mod (1)} =0$ almost surely and for every 
    $j \in [d]$ the identity 
    $$
     X_j = 1 - \left(\left(\sum_{i\neq j} X_i\right)  \,\, \textup{mod (1)}\right)
    $$
    holds almost surely. As a direct consequence, the copula $B$ of $\mathbf{X}=(X_1,\ldots,X_d)$ is
    exchangeable and fulfills $D_1(B,\Pi_d)=\frac{1}{3}=\Delta_{D_1}$. 
    The right panel of Figure \ref{fig1} depicts a sample (plus bivariate marginal samples) 
    of the copula $B$ for $d=3$.
\end{remark}

We now turn our attention to $\Delta_{d_\infty}$ and first derive a pointwise lower and upper bound for any copula $A \in \kc_d^{\Pi_{d-1}}$. Such bounds have already been considered in \cite{Deh83}, but the bounds derived 
in Proposition~\ref{prop:dinf} are stronger (they are, in fact, best possible). 
Furthermore, the book \cite{Joe97} discusses bounds for $d$-variate distribution functions with given $(d-1)$-variate marginals in the case $d=3,4$. In case the marginals are uniform, the formulas simplify to the bounds in Proposition~\ref{prop:dinf}.

\begin{proposition}\label{prop:dinf}
    Let $d \ge 2$, $A \in \kc_d^{\Pi_{d-1}}$ and $\mathbf{u} \in [0,1]^d$. Then
    $$ S_d(\mathbf{u}) \le A(\mathbf{u}) \le T_d(\mathbf{u}) ,$$
    where
    \begin{align*}
        S_d(\mathbf{u}) &= \Pi_d(\mathbf{u})-\min_{\substack{J \subseteq [d]\\ |J| \textup{ even}}} \Big\{\prod_{j \in J} (1-u_j) \prod_{j \in [d] \setminus J} u_j\Big\}, \text{ and}\\
        T_d(\mathbf{u}) &= \Pi_d(\mathbf{u})+\min_{\substack{J \subseteq [d]\\ |J| \textup{ odd}}} \Big\{\prod_{j \in J} (1-u_j) \prod_{j \in [d] \setminus J} u_j\Big\}.
    \end{align*}
    Moreover, for every $\mathbf{u} \in [0,1]^d$ there exist copulas $A_1,A_2 \in \kc_d^{\Pi_{d-1}}$ such that $A_1(\mathbf{u}) = S_d(\mathbf{u})$  and $A_2(\mathbf{u}) = T_d(\mathbf{u})$.
\end{proposition}

\begin{proof}
For every subset $J \subseteq [d]$ and every index $j \in [d]$ we denote by $E_j^J$ the interval $[u_j,1]$ if $j \in J$ and the interval $[0, u_j]$ if $j \in [d] \setminus J$. Furthermore, let $B_J$ be the rectangle $B_J = \displaystyle\bigtimes_{j \in [d]} E_j^J$ and let 
$$\textup{ver}(B_J) \subseteq \Delta := \bigtimes_{j \in [d]} \{0, u_j, 1\}$$ 
be the set of vertices of $B_J$. Note that $\mathbf{u} \in \textup{ver}(B_J)$ and $\textup{sgn}_{B_J}(\mathbf{u}) = (-1)^{|J|}$. Since $A \in \kc_d^{\Pi_{d-1}}$, we have $A(\mathbf{v}) = \Pi_d(\mathbf{v})$ for every
$\mathbf{v} \in \textup{ver}(B_J)$, except for $\mathbf{v} = \mathbf{u}$. 
This yields
\begin{align*}
    0 \le \mu_A(B_J) &= \sum_{\mathbf{v} \in \textup{ver}(B_J)} \textup{sgn}_{B_J}(\mathbf{v})A(\mathbf{v}) \\
     &= \sum_{\mathbf{v} \in \textup{ver}(B_J)} \textup{sgn}_{B_J}(\mathbf{v})\Pi_d(\mathbf{v}) - \textup{sgn}_{B_J}(\mathbf{u})\Pi_d(\mathbf{u}) + \textup{sgn}_{B_J}(\mathbf{u})A(\mathbf{u}) \\
     &= \mu_{\Pi_d}(B_J) - (-1)^{|J|}\Pi_d(\mathbf{u}) + (-1)^{|J|}A(\mathbf{u}) \\
     &= \prod_{j \in J} (1-u_j) \prod_{j \in [d] \setminus J} u_j - (-1)^{|J|}\Pi_d(\mathbf{u}) + (-1)^{|J|}A(\mathbf{u}).
\end{align*}
If $|J|$ is odd, it follows that
$$A(\mathbf{u}) \le \Pi_d(\mathbf{u}) + \prod_{j \in J} (1-u_j) \prod_{j \in [d] \setminus J} u_j,$$
if $|J|$ is even, we obtain
$$A(\mathbf{u}) \ge \Pi_d(\mathbf{u}) - \prod_{j \in J} (1-u_j) \prod_{j \in [d] \setminus J} u_j.$$
Collecting all obtained inequalities, we infer
$$A(\mathbf{u}) \le \min_{\substack{J \subseteq [d]\\ |J| \textup{ odd}}} \Big\{\Pi_d(\mathbf{u}) + \prod_{j \in J} (1-u_j) \prod_{j \in [d] \setminus J} u_j\Big\} = T_d(\mathbf{u}),$$
as well as 
$$A(\mathbf{u}) \ge \max_{\substack{J \subseteq [d]\\ |J| \textup{ even}}} \Big\{\Pi_d(\mathbf{u}) - \prod_{j \in J} (1-u_j) \prod_{j \in [d] \setminus J} u_j\Big\} = S_d(\mathbf{u}).$$
It remains to show that the bounds are attained. Let $\mathbf{u} \in [0,1]^d$ be arbitrary but fixed and define a function $\widetilde{A}_1 \colon \Delta \to [0,1]$ by 
$$\widetilde{A}_1(\mathbf{v}) = \begin{cases} S_d(\mathbf{u}); & \text{if } \mathbf{v} = \mathbf{u}, \\
\Pi_d(\mathbf{v}); & \text{otherwise.} 
\end{cases}$$
For every subset $J \subseteq [d]$ we have
\begin{align*}
\sum_{\mathbf{v} \in \textup{ver}(B_J)} \textup{sgn}_{B_J}(\mathbf{v})\widetilde{A}_1(\mathbf{v}) 
     &= \hspace{-3mm} \sum_{\mathbf{v} \in \textup{ver}(B_J)} \textup{sgn}_{B_J}(\mathbf{v})\Pi_d(\mathbf{v}) - \textup{sgn}_{B_J}(\mathbf{u})(\Pi_d(\mathbf{u}) - S_d(\mathbf{u})) \\
     &= \prod_{j \in J} (1-u_j) \prod_{j \in [d] \setminus J} u_j - (-1)^{|J|}(\Pi_d(\mathbf{u}) - S_d(\mathbf{u})).
\end{align*}
If $|J|$ is even, it follows that
$$\sum_{\mathbf{v} \in \textup{ver}(B_J)} \textup{sgn}_{B_J}(\mathbf{v})\widetilde{A}_1(\mathbf{v}) 
= \prod_{j \in J} (1-u_j) \prod_{j \in [d] \setminus J} u_j - \Pi_d(\mathbf{u}) + S_d(\mathbf{u}) \ge 0$$
by the definition of $S_d$. If  $|J|$ is odd, we have 
\begin{align*}
\sum_{\mathbf{v} \in \textup{ver}(B_J)} \textup{sgn}_{B_J}(\mathbf{v})\widetilde{A}_1(\mathbf{v}) 
&= \prod_{j \in J} (1-u_j) \prod_{j \in [d] \setminus J} u_j + \Pi_d(\mathbf{u}) - S_d(\mathbf{u}) \\
&\ge \prod_{j \in J} (1-u_j) \prod_{j \in [d] \setminus J} u_j \ge 0,    
\end{align*}
since $S_d(\mathbf{u}) \le \Pi_d(\mathbf{u})$. It follows that $\widetilde{A}_1$ is a subcopula which we can extend to a copula $A_1$, defined on entire $[0,1]^d$, via (piecewise) multilinear extension on each rectangle $B_J$ (see \cite{Nels} for the definition of subcopulas and their extensions). Since $\widetilde{A}_1(\mathbf{v}) = \Pi_d(\mathbf{v})$ for every $\mathbf{v} \in \Delta$ lying on the boundary of $[0,1]^d$, the $(d-1)$-dimensional marginals of $A_1$ are equal to $\Pi_{d-1}$, so $A_1 \in \kc_d^{\Pi_{d-1}}$. Furthermore, $A_1(\mathbf{u}) = S_d(\mathbf{u})$ by construction. Finally, defining a function $\widetilde{A}_2 \colon \Delta \to [0,1]$ by 
$$\widetilde{A}_2(\mathbf{v}) = \begin{cases} T_d(\mathbf{u}); & \text{if } \mathbf{v} = \mathbf{u}, \\
\Pi_d(\mathbf{v}); & \text{otherwise,} 
\end{cases}$$
showing that it is a subcopula similarly as above, and extending it to a copula $A_2 \in \kc_d^{\Pi_{d-1}}$, satisfying $A_2(\mathbf{u}) = T_d(\mathbf{u})$, completes the proof.
\end{proof}

Note that in the case $d=2$ the bounds $S_2$ and $T_2 $ are equal to the Fr\'{e}chet-Hoeffding bounds $W_2$ and $M_2$. In the case $d=3$ the bounds in Proposition~\ref{prop:dinf} coincide with functions $S_3$ and $T_3$ given in \cite{NeUb12}. Furthermore, being pointwise infima and suprema of copulas, the functions $S_d$ and $T_d$ are quasi-copulas and, for $d \ge 3$, not copulas, as the following lemma shows. For $d=3$ this was already proved in \cite{NeUb12}.

\begin{lemma}
Let $d \geq 3$ and consider the rectangle $B = [\frac13, \frac23]^3 \times [0, \frac13]^{d-3}$ in $[0, 1]^d$. Then 
$$ \mu_{S_d}(B) = \mu_{T_d}(B) = - \tfrac{1}{3^{d-1}} < 0$$
holds, so both $S_d$ and $T_d$ are proper quasi-copulas.
\end{lemma}

\begin{proof}
Consider
    \begin{align*}
        \widetilde{S}_d(\mathbf{u}) &:= \Pi_d(\mathbf{u}) - S_d(\mathbf{u}) = \min_{\substack{J \subseteq [d]\\ |J| \textup{ even}}} \Big\{\prod_{j \in J} (1-u_j) \prod_{j \in [d] \setminus J} u_j\Big\},\\
        \widetilde{T}_d(\mathbf{u}) &:= T_d(\mathbf{u}) - \Pi_d(\mathbf{u}) = \min_{\substack{J \subseteq [d]\\ |J| \textup{ odd}}} \Big\{\prod_{j \in J} (1-u_j) \prod_{j \in [d] \setminus J} u_j\Big\}.
    \end{align*}
and set
    \begin{align*}
        \mathbf{u}_3 &:= (\tfrac23, \tfrac23, \tfrac23, \tfrac13, \ldots, \tfrac13), \\
        \mathbf{u}_2 &:= (\tfrac23, \tfrac23, \tfrac13, \tfrac13, \ldots, \tfrac13), \\
        \mathbf{u}_1 &:= (\tfrac23, \tfrac13, \tfrac13, \tfrac13, \ldots, \tfrac13), \text{ and}\\
        \mathbf{u}_0 &:= (\tfrac13, \tfrac13, \tfrac13, \tfrac13, \ldots, \tfrac13)
    \end{align*}
Then all four points $\mathbf{u}_0,\ldots,\mathbf{u}_3$ are vertices of the 
rectangle $B$ and all other vertices of $B$ are either obtained by permuting the first three coordinates of these vertices or contain at least one coordinate equal to $0$. We claim that
$$ \widetilde{S}_d(\mathbf{u}_3) = \tfrac{2}{3^d}, \quad \widetilde{S}_d(\mathbf{u}_2) = \tfrac{1}{3^d}, \quad 
        \widetilde{S}_d(\mathbf{u}_1) = \tfrac{2}{3^d}, \quad  \text{and} \quad \widetilde{S}_d(\mathbf{u}_0) = \tfrac{1}{3^d}. $$
Indeed, this is straightforward for $d=3$. If $d > 3$, the minimum over sets $J \subseteq [d]$ of even cardinality
is always achieved by a set $J$ with $J \cap \{4, ..., d\} = \emptyset$, since including any coordinate from $\{4, ..., d\}$
in $J$ replaces a factor of $\tfrac13$ by $\tfrac23$, so strictly increases the product. So, the value of $\widetilde{S}_d$ at $\mathbf{u}_k$ is obtained by multiplying the value of $\widetilde{S}_3$ at corresponding $\mathbf{u}_k$ by $\tfrac{1}{3^{d-3}}$.
It follows that 
$$ \mu_{\widetilde{S}_d}(B) = \widetilde{S}_d(\mathbf{u}_3) - 3\widetilde{S}_d(\mathbf{u}_2) + 3\widetilde{S}_d(\mathbf{u}_1) - \widetilde{S}_d(\mathbf{u}_0) = \tfrac{4}{3^d},$$
since the remaining vertices of $B$ contain a zero coordinate, hence do not contribute to the
signed volume. This implies
$$ \mu_{S_d}(B) = \mu_{\Pi_d}(B) - \mu_{\widetilde{S}_d}(B) = \tfrac{1}{3^d} - \tfrac{4}{3^d} = -\tfrac{1}{3^{d-1}}.$$
Using similar arguments as above, we obtain
$$ \widetilde{T}_d(\mathbf{u}_3) = \tfrac{1}{3^d}, \quad \widetilde{T}_d(\mathbf{u}_2) = \tfrac{2}{3^d}, \quad 
        \widetilde{T}_d(\mathbf{u}_1) = \tfrac{1}{3^d}, \quad  \text{and} \quad \widetilde{T}_d(\mathbf{u}_0) = \tfrac{2}{3^d}, $$
so $\mu_{\widetilde{T}_d}(B) = -\tfrac{4}{3^d}$, and $\mu_{T_d}(B) = \mu_{\Pi_d}(B) + \mu_{\widetilde{T}_d}(B) = -\tfrac{1}{3^{d-1}}.$
\end{proof}

Now, we compute the value of $\Delta_{d_\infty}$ for every $d$.

\begin{theorem}\label{thm:maindinf}
For every $d \geq 2$ we have $\Delta_{d_\infty}=\frac{1}{2^d}$ and the supremum is attained.
\end{theorem}

\begin{proof}
For every $\mathbf{u} \in [0,1]^d$, using inequality $\min\{a_1, a_2, \ldots, a_m\} \le \sqrt[m]{\prod_{i=1}^m a_i}$, we have   
\begin{align*}
    T_d(\mathbf{u}) - \Pi_d(\mathbf{u}) &= \min_{\substack{J \subseteq [d]\\ |J| \textup{ odd}}} \Big\{\prod_{j \in J} (1-u_j) \hspace{-2mm} \prod_{j \in [d] \setminus J} u_j\Big\} 
    \le \hspace{-2mm} \sqrt[2^{d-1}]{\prod_{\substack{J \subseteq [d]\\ |J| \textup{ odd}}} \hspace{-2mm} \Big(\prod_{j \in J} (1-u_j) \hspace{-2mm} \prod_{j \in [d] \setminus J} u_j\Big)}.
\end{align*}
Note that every $j \in [d]$ appears in precisely $2^{d-2}$ subsets $J \subseteq [d]$ of odd cardinality (and does not appear in the remaining $2^{d-2}$ subsets $J \subseteq [d]$ of odd cardinality). This implies
\begin{align*}
     T_d(\mathbf{u}) - \Pi_d(\mathbf{u})    &\le \sqrt[2^{d-1}]{\prod_{j \in [d]} (1-u_j)^{2^{d-2}} \prod_{j \in [d]} u_j^{2^{d-2}}} = \sqrt{\prod_{j \in [d]} (1-u_j)u_j} \le \tfrac{1}{2^d},
\end{align*}
since $(1-u_j)u_j \le \tfrac14$ for every $j$. Similarly, for every $\mathbf{u} \in [0,1]^d$
$$\Pi_d(\mathbf{u}) - S_d(\mathbf{u}) = \min_{\substack{J \subseteq [d]\\ |J| \textup{ even}}} \Big\{\prod_{j \in J} (1-u_j) \prod_{j \in [d] \setminus J} u_j\Big\} \le \tfrac{1}{2^d}.$$
It hence follows that
\begin{align*}
\Delta_{d_\infty} &= \sup\Big\{d_\infty(A,\Pi_d):\ A \in \kc_d^{\Pi_{d-1}}\Big\} \\
&= \sup\Big\{\sup\{|A(\mathbf{u}) - \Pi_d(\mathbf{u})|:\ \mathbf{u} \in [0,1]^d\}:\ A \in \kc_d^{\Pi_{d-1}}\Big\}  \\
&= \sup\Big\{\sup\big\{|A(\mathbf{u}) - \Pi_d(\mathbf{u})|:\ A \in \kc_d^{\Pi_{d-1}}\big\}:\ \mathbf{u} \in [0,1]^d\Big\}  \\
&= \sup\Big\{\max\{T_d(\mathbf{u}) - \Pi_d(\mathbf{u}), \Pi_d(\mathbf{u}) - S_d(\mathbf{u})\}:\ \mathbf{u} \in [0,1]^d\Big\}  \\
&\le \tfrac{1}{2^d}.
\end{align*}
On the other hand, let $\mathbf{u}_0 = (\tfrac12, \tfrac12, \ldots, \tfrac12) \in [0,1]^d$, to estimate
\begin{align*}
\Delta_{d_\infty} &= \sup\Big\{\max\{T_d(\mathbf{u}) - \Pi_d(\mathbf{u}), \Pi_d(\mathbf{u}) - S_d(\mathbf{u})\}:\ \mathbf{u} \in [0,1]^d\Big\}  \\
&\ge T_d(\mathbf{u}_0) - \Pi_d(\mathbf{u}_0) = \min_{\substack{J \subseteq [d]\\ |J| \textup{ odd}}} \Big\{\prod_{j \in J} (1-\tfrac12) \prod_{j \in [d] \setminus J} \tfrac12 \Big\} \\
&= \min_{\substack{J \subseteq [d]\\ |J| \textup{ odd}}} \{\tfrac{1}{2^d} \} = \tfrac{1}{2^d}.
\end{align*}
This shows that the supremum is attained and finishes the proof.
\end{proof}

\begin{figure}[ht]
\begin{center}
        \includegraphics[width=1\textwidth]{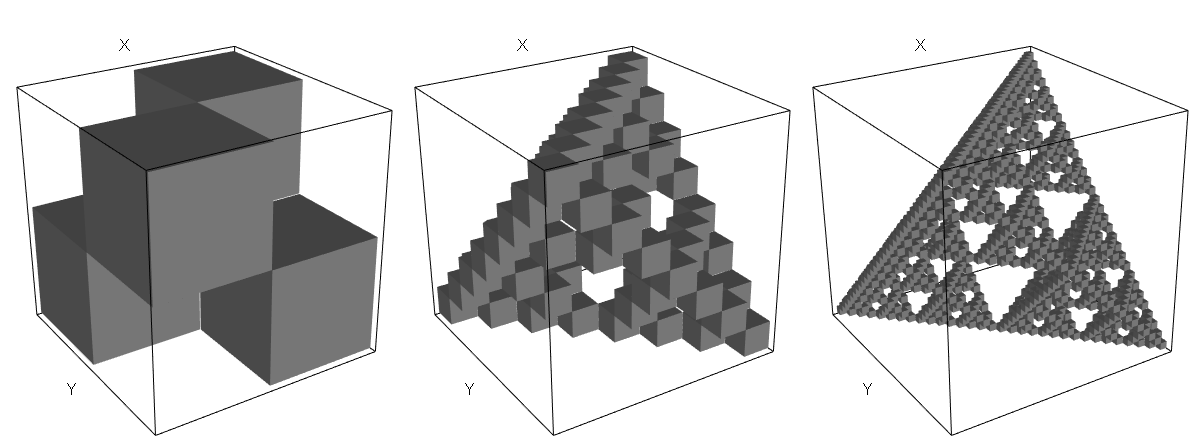}
\end{center}
    \caption{Steps 1, 3 and 5 in the IFSP-based construction of the 
    completely dependent copula $A \in \kc_3^{\Pi_2}$ fulfilling 
    $d_\infty(A,\Pi_3)=\frac{1}{2^3}=\Delta_{d_\infty}$ as well as 
    $D_1(A,\Pi_3)=\frac{1}{3}=\Delta_{D_1}$.}
    \label{fig2}
\end{figure}

\begin{example}
    Although the metrics $d_\infty$ and $D_1$ are conceptually quite different, 
    there are copulas $A \in \kc_d^{\Pi_{d-1}}$ fulfilling that they maximize both $\Delta_{d_\infty}$ and 
    $\Delta_{D_1}$. In fact, considering $d=3$ and working with Iterated Function Systems 
    with Probabilities (IFSP) induced by so-called generalized transformation 
    matrices, see \cite{DFT},
    it is possible to construct a copula $A \in \kc_3^{\Pi_2}$ which is 
    completely dependent and which fulfills $A(\frac{1}{2},\frac{1}{2},\frac{1}{2})=0$, implying $d_\infty(A,\Pi_3) \geq \vert 0-\frac{1}{2^3} \vert =\frac{1}{2^3}=\Delta_{d_\infty}$. Figure \ref{fig2} depicts the stepwise IFSP  
    approximation of support/density of this copula $A$. 
\end{example}

\begin{remark}
The conceptual difference of the metrics $d_\infty$ and $D_1$ shows itself
once again in the fact that 
$\Delta_{d_\infty}=\frac{1}{2^d} \rightarrow 0$ for $d \rightarrow \infty$ 
whereas $\Delta_{D_1}=\frac{1}{3}$ for every $d \geq 3$. Intuitively, 
with growing dimension $d$, the growing number of $(d-1)$-marginal constraints 
in $\kc_d^{\Pi_{d-1}}$ allows for less and less flexibility in the sense of pointwise deviation
from $\Pi_d$ (in fact, the flexibility reduces exponentially). 
By contrast, as mentioned in Section \ref{sec:2}, $D_1(A,\Pi_d)$ is maximal if and only if $A$ is completely dependent and the maximal value is given by 
$\frac{1}{3}$ for every $d \geq 3$. 
Since complete dependence is compatible with uniform $(d-1)$-marginals, 
$\kc_d^{\Pi_{d-1}}$ contains completely dependent elements and
the maximum possible $D_1$-distance to 
$\Pi_d$ does not decay with growing dimension $d$. 
\end{remark}

\section*{Acknowledgments}

Damjana Kokol Bukovšek and Nik Stopar acknowledge financial support from the ARIS (Slovenian Research and Innovation Agency, research core funding No. P1-0222 and projects J1-70034 and J1-50002). Wolfgang Trutschnig gratefully acknowledges the support 
of the WISS 2025 project ‘IDA-lab Salzburg’ (20204-WISS/225/197-2019 and 20102-F1901166 KZP).

\bibliographystyle{amsplain}
\bibliography{biblio}

\end{document}